\documentclass[letterpaper,10pt,reqno]{amsart}
\usepackage{indentfirst} 
\usepackage{amssymb}
\usepackage{mathtools} 
\usepackage{mathabx} 
\usepackage{amsthm}
\usepackage{thmtools}
\usepackage{enumitem} 
\usepackage[colorlinks=true]{hyperref} 
\usepackage[usenames,dvipsnames]{xcolor} 
\usepackage{ifthen}
\usepackage{tikz}
\usetikzlibrary{decorations.pathreplacing}
\usepackage{tikz-cd}
\usepackage[utf8]{inputenc}
\usepackage{qtree}
\usepackage{float}
\usepackage[justification=centering]{caption}
  
\makeatletter

\def\MRbibitem{\@ifnextchar[\my@lbibitem\my@bibitem}

\def\mybiblabel#1#2{\@biblabel{{\hyperref{http://www.ams.org/mathscinet-getitem?mr=#1}{}{}{#2}}}}

\def\myhyperanchor#1{\Hy@raisedlink{\hyper@anchorstart{cite.#1}\hyper@anchorend}}

\def\my@lbibitem[#1]#2#3#4\par{%
  \item[\mybiblabel{#2}{#1}\myhyperanchor{#3}\hfill]#4%
  \@ifundefined{ifbackrefparscan}{}{\BR@backref{#3}}%
  \if@filesw{\let\protect\noexpand\immediate
    \write\@auxout{\string\bibcite{#3}{#1}}}\fi\ignorespaces%
}

\def\my@bibitem#1#2#3\par{%
  \refstepcounter\@listctr
  \item[\mybiblabel{#1}{\the\value\@listctr}\myhyperanchor{#2}\hfill]#3%
  \@ifundefined{ifbackrefparscan}{}{\BR@backref{#2}}%
  \if@filesw\immediate\write\@auxout
    {\string\bibcite{#2}{\the\value\@listctr}}\fi\ignorespaces%
}

\makeatother

\DeclareFontFamily{U} {MnSymbolA}{}
\DeclareFontShape{U}{MnSymbolA}{m}{n}{
   <-6> MnSymbolA5
   <6-7> MnSymbolA6
   <7-8> MnSymbolA7
   <8-9> MnSymbolA8
   <9-10> MnSymbolA9
   <10-12> MnSymbolA10
   <12-> MnSymbolA12}{}
\DeclareFontShape{U}{MnSymbolA}{b}{n}{
   <-6> MnSymbolA-Bold5
   <6-7> MnSymbolA-Bold6
   <7-8> MnSymbolA-Bold7
   <8-9> MnSymbolA-Bold8
   <9-10> MnSymbolA-Bold9
   <10-12> MnSymbolA-Bold10
   <12-> MnSymbolA-Bold12}{}
\DeclareSymbolFont{MnSyA} {U} {MnSymbolA}{m}{n}
 \DeclareFontFamily{U} {MnSymbolC}{}
\DeclareFontShape{U}{MnSymbolC}{m}{n}{
  <-6> MnSymbolC5
  <6-7> MnSymbolC6
  <7-8> MnSymbolC7
  <8-9> MnSymbolC8
  <9-10> MnSymbolC9
  <10-12> MnSymbolC10
  <12-> MnSymbolC12}{}
\DeclareFontShape{U}{MnSymbolC}{b}{n}{
  <-6> MnSymbolC-Bold5
  <6-7> MnSymbolC-Bold6
  <7-8> MnSymbolC-Bold7
  <8-9> MnSymbolC-Bold8
  <9-10> MnSymbolC-Bold9
  <10-12> MnSymbolC-Bold10
  <12-> MnSymbolC-Bold12}{}
\DeclareSymbolFont{MnSyC} {U} {MnSymbolC}{m}{n}

\DeclareMathSymbol{\top}{\mathord}{MnSyA}{219} 
\DeclareMathSymbol{\plus}{\mathord}{MnSyC}{20} 

\declaretheorem[numberwithin=section]{theorem}
\declaretheorem[sibling=theorem]{lemma}
\declaretheorem[sibling=theorem]{corollary}
\declaretheorem[sibling=theorem]{proposition}
\declaretheorem[sibling=theorem,style=definition]{definition}

\declaretheorem[sibling=theorem,style=remark]{remark}

\hypersetup{bookmarksdepth = 3} 
\numberwithin{equation}{section}     

\setlist[enumerate,1]{label={\upshape(\alph*)},ref=\alph*}
\setlist[enumerate,2]{label={\upshape(\arabic*)},ref=\arabic*}

\newcommand{\R}{\mathbb{R}}

\newcommand{\N}{\mathbb{N}}

\def\phi{\varphi}
\def\R{{\mathbb R}}

\def\N{{\mathbb N}}

\def\Q{{\mathbb Q}}

\newcommand{\vertiii}[1]{{\left\vert\kern-0.25ex\left\vert\kern-0.25ex\left\vert #1 
    \right\vert\kern-0.25ex\right\vert\kern-0.25ex\right\vert}}
\newcommand{\invertiii}[1]{{\vert\kern-0.25ex\vert\kern-0.25ex\vert #1 
    \vert\kern-0.25ex\vert\kern-0.25ex\vert}}

\begin{document}

\title{Odometers in non-compact spaces}



\begin{thanks}  {We would like to thank An\'ibal Velozo for interesting discussions on the topic of this article and  Giovanni Panti for informing us of \cite{bi}. GI was partially supported by Proyecto Fondecyt 1230100. MP was partially supported by Proyecto Fondecyt 1220032}
 \end{thanks}

\author[G.~Iommi]{Godofredo Iommi} \address{Facultad de Matem\'aticas,
Pontificia Universidad Cat\'olica de Chile (UC), Avenida Vicu\~na Mackenna 4860, Santiago, Chile}
 \email{\href{mailto:godofredo.iommi@gmail.com}{godofredo.iommi@gmail.com}} 
\urladdr{\href{http://www.mat.uc.cl/~giommi}{www.mat.uc.cl/$\sim$giommi}}

 \author[M.~Ponce]{Mario Ponce}   \address{Facultad de Matem\'aticas,
Pontificia Universidad Cat\'olica de Chile (UC), Avenida Vicu\~na Mackenna 4860, Santiago, Chile}
\email{\href{mponcea@mat.uc.cl}{mponcea@mat.uc.cl}}

\begin{abstract}
We define an \emph{odometer} in the Baire space. That is the non-compact space of one sided sequences of natural numbers. We go on to prove that it is topologically conjugated to the dyadic odometer restricted to an appropriate non-compact subset of the space of one sided sequences on two symbols. We extend the action of the odometer to finite words. It turns out that, arranging  the finite words in appropriate binary tree, this extension runs through the tree from left-right and top-down. Several well known binary trees made out of rational numbers, such as  the Kepler tree or the Calkin and Wilf tree, are recovered in this way. Indeed, it suffices to identify  finite words with rational numbers by means of certain continued fractions. The odometric action allow us to recover, in a unified way, several counting results for the rational numbers.  Moreover, associated to certain interval maps with countably many branches we construct their corresponding odometers. Explicit formulas are provided in the cases of the Gauss map (continued fractions) and the Renyi map (backward continued fractions).
\end{abstract}

\maketitle

\section{Introduction}
Odometers form an important class of dynamical systems. They capture and exemplify how complicated simple dynamical systems can be.  A well known example is the \emph{dyadic odometer}. In this case, the phase space is that of one sided sequences of zeroes and ones. The action is given by the addition by $\underline{1}=(1,0,0, \dots)$ with the \emph{carry over} rule (see section \ref{sec:def} for details). For this system,  every orbit is dense and it has a unique invariant measure. Interestingly, every continuous interval map $f:[0,1] \to [0,1]$ of zero topological entropy with an invariant non-atomic measure is metrically isomorphic to the dyadic odometer (see, for example, \cite[Theorem 15.4.2]{kh}).  The dyadic odometer is related to binary expansions. More general odometers  related to other numeration systems, such as  Cantor series,  can also be constructed \cite{glt}. We refer to the survey \cite{d} by Downarowicz, for a wealth of properties and examples. In all the above settings, the phase space is compact. In this note we provide a simple and explicit definition of an \emph{odometer} in the Baire space $\N^{\N}$, see Definition \ref{def_odo}. This space is endowed with the natural product topology and it is a non compact space (actually, not even locally compact). Our definition is similar to the \emph{addition by one with carry over} but with countably many symbols. The similarity is not only formal, we actually prove that it is topologically conjugated to the dyadic odometer restricted an appropriate non-compact subset of  the space of one sided sequences on two symbols (see Theorem \ref{main_odo}). Therefore, the map we define is minimal, uniquely ergodic, has zero topological entropy, pure discrete spectrum and, of course, all other properties of the dyadic odometer that are preserved under topological conjugacy. Recall that the Baire space is dynamical model for a wide range of continued fractions. Therefore, this enable us to define odometers in numerical systems with countably many symbols.
 We note that, recently, certain substitution systems have been defined in countable alphabets \cite{dmfv,f}.

Odometers are counting devices. Perhaps, one of the first manifestation of this is the counting of dyadic rationals. Indeed, it is possible to extend the action of the dyadic odometer to finite words. The binary expansion of real numbers allows for the identification of finite words in the alphabet $\{0,1\}$  with rational dyadic numbers. It turns out that the orbit of the word $(1)$, with respect to the dyadic odometer, provides a bijection between the sets of dyadic numbers and $\N \cup \{0\}$. A standard way to see this is to order the dyadic rationals in a binary tree and to define a function that runs through the tree from left-right and top-down. This is can be achieved using the dyadic odometer as explained in section \ref{dyd}. We extend the action of the odometer on the Baire space to  finite words on a countable alphabet, see section \ref{f}. Several continued fraction systems allow us to identify finite words on a countable alphabet with rational numbers. Proceeding as in the case of the dyadic tree, it is possible to (re)obtain several well known trees of rationals. For example, the Kepler tree,  which was defined by Kepler in 1619, can be obtained in this way identifying finite words with rational numbers using the continued fraction expansion, see section \ref{con_f}. Also, the Calkin and Wilf tree, see section \ref{con_b}, is obtained using the same method and  identifying finite words with rational numbers using the backward continued fraction expansion. Several results on counting rationals that is, providing explicit bijections between $\N \cup \{0\}$ and rationals in $(0,1)$, can now be recovered just considering the orbit of the root of the tree with respect to the appropriate extension of the action of the odometer to the set of finite words.

The Baire space has strong universality properties. For example, Alexandrov and Urysohn  proved that the Baire space is the unique, up to homeomorphism, nonempty zero-dimensional Polish  space, for which all compact subsets have empty interior. This, and similar, properties allows for the extension of the definition of \emph{non-compact} odometer to a wide range of spaces. For example, it is possible to define odometers in non-locally compact symbolic spaces different from the Baire space. These spaces occur naturally in the coding of certain dynamical systems. 

More concretely, for a class of interval maps with countably many branches we define the corresponding odometer. We consider first the case of the Gauss map, related to the continued fraction expansion. We obtain an explicit formula  for its odometer which is defined in terms of the Fibonacci numbers (Proposition \ref{action-gauss}). Also, we study the odometer associated to the Renyi map, which is related to the backward continued fraction expansion. It turns out that the corresponding odometer is a map that has been used to construct explicit bijections between the natural and the rational numbers (Proposition \ref{bcf}). Bonanno and Isola \cite[Theorem 2.3]{bi} proved that this map is topologically conjugated to the dyadic odometer and we  recovered the same result using backward continued fractions in \cite{ip}.

\section{Odometer in the Baire space} \label{sec:def}
In this section we define an odometer in a non-compact symbolic space. We prove that it shares all the dynamically relevant properties of its compact counterpart. Endow the set $\{0,1\}$ with the discrete topology and $\Sigma_2=\{0,1\}^{\N}$, with the product topology. This is a compact metrizable space.

 \begin{definition}
The  {\em dyadic odometer} is the map $O: \Sigma_2 \to \Sigma_2$ defined by
\begin{equation*}
O((\underbrace{1,\dots, 1,}_{k\textrm{-times}}0,w_{k+2},w_{k+3},\dots))=(\underbrace{0,\dots, 0,}_{k\textrm{-times}}1,w_{k+2},w_{k+3},\dots),
\end{equation*}
and $O((1,1, \dots))=(0,0, \dots )$. 
\end{definition}

The map $O$ is also known as von Neumann-Kakutani odometer, adding machine or $2-$odometer. The dynamics of this system is well understood, see \cite[pp.25-26]{aa} and \cite[Section 4.5.2]{b}. We enumerate a sample of properties:

\begin{proposition} \label{erg_odo}
The dyadic odometer satisfies:
\begin{enumerate}
\item \label{1} For every $n \in \N$ and $w \in \Sigma_2$ we have,
\begin{equation*}
\left\{ ((O^k w)_1, \dots , (O^k w)_n ): 0 \leq k \leq 2^n -1 \right\} = \{0,1\}^n.
\end{equation*}
 \item The map $O$ is minimal.
 \item The map $O$ is uniquely ergodic and its unique invariant measure is the $(1/2, 1/2)-$Bernoulli measure.
\item The map $O$ has zero topological entropy.
\item The map $O$ has purely discrete spectrum.
\item The map $O$ is equicontinuous.
\end{enumerate}
\end{proposition}

%

Let $\N$ be the set of natural numbers, $\N_0= \N \cup\{0\}$ and $\Sigma_{\geq 0}= \N_0^{\N}$. Endow $\N_0$ with the discrete topology and $\Sigma_{\geq 0}$ with the product topology. The space $\Sigma_{\geq 0}$ is not compact nor locally-compact. As in descriptive set theory, the space $\Sigma_{\geq 0}$ is called  \emph{Baire set}.

\begin{definition} \label{def_odo}
The $0-$\emph{odometer} is the map $O_0: \Sigma_{\geq 0} \to \Sigma_{\geq 0}$ defined by
\begin{equation*}
O_0( (w_1, w_2, w_3,  \dots) )=(\underbrace{0,\dots,0,}_{w_1\textrm{-times}}(w_2+1), w_3, \dots).
\end{equation*}
\end{definition}

We will prove that the map $O_0$ is topologically conjugated to the restriction of  $O$ to an appropriate non-compact subset of $\Sigma_2$. Let 
\begin{eqnarray*}
\Sigma_{\bar1}= \left\{(w_i)_i \in \Sigma_2 : \text{ there exists } k \in \N \text{ s.t. for every } i \geq k, w_i=1 \right\}, &\\
\Sigma_{\bar0}= \left\{(w_i)_i \in \Sigma_2 : \text{ there exists } k \in \N \text{ s.t. for every } i \geq k, w_i=0 \right\}.\end{eqnarray*}
Also consider $\Sigma_{2,0}=\Sigma_2 \setminus \Sigma_{\bar1}$ and endow it with the induced topology.

\begin{theorem} \label{main_odo}
There exits an homemorphism $\eta: \Sigma_{2,0} \mapsto \Sigma_{\geq 0}$ such that
\begin{equation*}
O_0 \circ \eta = \eta \circ O.
\end{equation*}
 \end{theorem}

\begin{proof}
In order to prove this result it will be convenient to recode the space $\Sigma_{2,0}$. For every integer $k \geq 0$ we the define the block
\begin{eqnarray} \label{bk}
b_{k}=
\begin{cases}
0 & \text{ if  } k=0,  \\
\underbrace{1, \dots, 1}_{k\textrm{-times}},0 & \text{ if } k \geq 1.
\end{cases}
\end{eqnarray}
Every sequence $(w_i)_i \in \Sigma_{2,0}$ can be uniquely written as an infinite sequence in the alphabet $\{b_k :k\geq 0\}$. For example,
\begin{equation*}
(0,0,1,1,1,0,1,0,1,1,1,1,1,0,1,0,\dots)=(b_0,b_0,b_3,b_1,b_5,b_1,\dots).
\end{equation*}
Note that the  map $\eta: \Sigma_{2,0}\mapsto\Sigma_{\geq 0}$ defined by 
\begin{equation*}
\eta((w_1,w_2, \dots))= \eta (b_{k_1}, b_{k_2}, \dots)= (k_1,  k_2,  \dots),
\end{equation*}
is an homeomorphism.  The following holds
 \begin{equation}
(\eta\circ O \circ \eta^{-1}) (k_1, k_2, k_3, \dots)=(\underbrace{0,\dots,0,}_{k_1\textrm{-times}}(k_2+1),k_3,  \dots).
 \end{equation}
Indeed,  we have $\eta^{-1}(k_1, k_2, \dots)=(\underbrace{1,\dots, 1,}_{k_1\textrm{-times}}0, b_{k_2}, \dots)$. Thus, 

\[
(\eta\circ  O \circ \eta^{-1})\left(\underbrace{1,\dots, 1,}_{k_1\textrm{-times}}0, b_{k_2},\dots \right)=\left( \underbrace{0,\dots, 0}_{k_1\textrm{-times}},1,b_{k_2},\dots\right), 
\]
and the conclusion follows from the fact that $1,b_k=b_{k+1}$. That is, 
\begin{equation*}
O_0= (\eta\circ O \circ \eta^{-1}). 
\end{equation*}
\end{proof}
%

\begin{remark}
For every $k \in \N$, let $\N_k:=\{ n \in \N : n \geq k\}$ and $\Sigma_{\geq k}=\N_k^{\N}$ endowed with the product topology. The map  $O_k: \Sigma_{\geq k} \mapsto \Sigma_{\geq k}$ defined by
\begin{equation*}
O_k((w_1, w_2, w_3, \dots))=(\underbrace{k,\dots,k,}_{(w_1-k)\textrm{-times}}(w_2+1), w_3, \dots).
\end{equation*}
is topologically conjugated to $O_0$
\end{remark}
The following \emph{odometric} property follows from Theorem \ref{main_odo} and  Proposition \ref{erg_odo}.

\begin{remark}\label{barrido}
Given an infinite word $\omega=(w_1, w_2, w_3, \dots) \in \Sigma_{\geq 0}$ and any finite word  $(\bar{w_1}, \bar{w_2}, \dots, \bar{w_k} )$, with $\bar w_j\in \N$, there exists $n>0$ such that $O_0^n(\omega)=(\bar{w_1}\, \bar{w_2}, \dots, \bar{w_k}, \dots)$. That is, given a finite word, any infinite word has an iterate that starts exactly as the given finite word.   
\end{remark}

\begin{definition}
For every $k \in \N$, the \emph{shift} map is the transformation $\sigma:   \Sigma_{\geq k} \mapsto \Sigma_{\geq k}$ defined by
\begin{equation*}
\sigma((w_1, w_2, w_3 \dots))=(w_2, w_3 \dots).
\end{equation*}
\end{definition}

This map is somehow at the opposite dynamical end of the odometer. Indeed, it has infinite entropy, it is topologically mixing, and for every $n \in \N$ it has countable many periodic orbits of prime period $n$.

\subsection{Renormalization Property} \label{reno} The odometer and the shift map satisfy a renormalization property analogous to that satisfied by the geodesic and horocycle flows in the unit tangent bundle of a hyperbolic surface   (see \cite[Property 3.3 p.119]{da} or \cite{ss}). Indeed, denote by $\sigma_2: \Sigma_2 \to \Sigma_2$ the shift map on $\Sigma_2$. For every $m,n \in \N \cup \{0\}$ and $w \in \Sigma_2$  we have
\begin{equation} \label{g-h}
\left(O^m \circ \sigma_2^n \right) (w) =\left( \sigma_2^n \circ O^{m 2^n} \right) (w).
\end{equation}
The number $2$ in the exponent of $O$ in equation \eqref{g-h} is related to the fact that the topological entropy of  $\sigma_2$ is $\log 2$. In our next result we establish a formula analogous to that in equation \eqref{g-h} for the odometer $O_0$ and the full-shift on a countable alphabet  $\sigma: \Sigma_{\geq 0} \to \Sigma_{\geq 0}$. We stress that the later system has infinite entropy and that, in this case, the relation depends upon the point where it is evaluated. 

\begin{proposition}
For every $w=(k_1, k_2, \dots) \in \Sigma_{\geq 0}$ and $m,n \in \N \cup \{0\}$, we have
\begin{equation}
\left(O_0^m \circ \sigma^n\right)(w) = \left(\sigma^n \circ O_0^{m 2^n 2^{\sum_{i=1}^n k_i}} \right)(w).
\end{equation}\end{proposition}

\begin{proof}
Recall that, restricted to the appropriate non-compact subspace, we have $O= \eta^{-1} \circ  O_0 \circ \eta$. Replacing in equation \eqref{g-h} we obtain, for every $(k_1, k_2, \dots) \in \Sigma_{\geq 0}$ and $m,n \in \N \cup \{0\}$, that
\begin{equation} \label{nose}
O_0 \circ (\eta \circ \sigma_2^n \circ \eta^{-1}) = (\eta \circ \sigma_2^n \circ \eta^{-1}) \circ O_0^{m 2^n}.
\end{equation}
The map $\sigma$ can be thought of as the \emph{accelerated} version of $\sigma_2$ in the following sense:
let $w=(k_1, k_2, k_3, \dots ) \in \Sigma_{\geq 0}$, then
\begin{equation} \label{acce}
(k_2,k_3, \dots)= \sigma(w)=  (\eta \circ \sigma_2^{k_1+1} \circ \eta^{-1})(w).
\end{equation}
Combining equations \eqref{g-h}, \eqref{nose} and \eqref{acce}, we obtain the desired result.
\end{proof}

\section{Counting finite words} \label{f}

In this section we extend the action of the odometer to finite words. We prove that arranging  them in appropriate binary tree, the extension of the odometer runs through the tree from left-right and top-down. Identifying finite words with rational numbers by means of certain continued fractions, we recover  well known binary trees made out of rational numbers, such as  the Kepler tree or the Calkin and Wilf tree. Moreover, this extension of the odometer to finite words allow us to recover, in a unified way, several counting results for the rational numbers.

%

%
%

\subsection{Odometric binary tree} We will represent the set of {\it finite words}, that we denote by $\Omega_1= \bigcup_{j \geq 1} \N_1^j$, as a binary tree.

\begin{definition}  The {\it odometric binary tree}, denoted by $\mathcal{T}_1$, is a tree defined by the following rules:
\begin{enumerate}
\item The {\it leafs} of $\mathcal{T}_1$   are the elements in $\Omega_1$. 
\item Every leaf $\omega=(w_1, w_2, \dots, w_j) \in \Omega_1$ has two {\it sons} displayed in the following order:
\[
\omega^1=(1, w_1, w_2, \dots, w_j) \quad \textrm{and}\quad \omega^{+1}=(w_1+1, w_2, \dots, w_j).
\] 
\item The root of the tree is $(1) \in \Omega_1$. 
\end{enumerate}
Note that, apart from $(1)$, every leaf  is the son of exactly one leaf. The    {\it level} $1$ of the tree, denoted by $\mathcal{T}^1_1$, contains $(1)$ as its only element. Inductively, for $l \in \N$, we say that sons of a leaf  at  level $\mathcal{T}_1^l$  are at the level $\mathcal{T}_1^{l+1}$. 
\end{definition}

The first four levels of the odometric binary tree are displayed in Figure \ref{tree}.

\begin{figure}
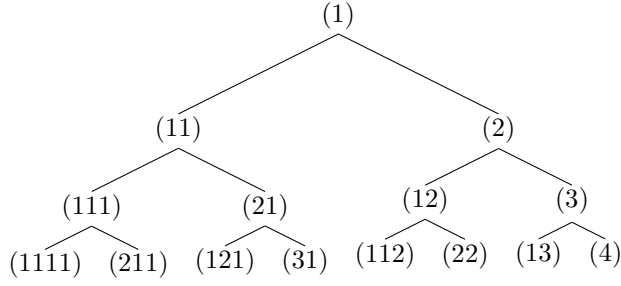

 \Tree [.(1) [.(11) [.(111) [.(1111) ] [.(211) ] ] [.(21) [.(121) ] [.(31) ] ] ] [.(2) [.(12) [.(112) ] [.(22) ] ] [.(3) [.(13) ] [.(4)  ] ] ]]
\centering
 \caption{The first four levels of the odometric binary tree or the  $\mathcal{T}_1$ tree.}
  \label{tree}
\end{figure}

\begin{definition}
For a finite word  $ \omega=(w_1, w_2, \dots, w_j)$, we define its \emph{sum}, that we denote $s(\omega)$,  by
\[
s(\omega)=w_1+w_2+\dots+w_j.
\]
For $l \in \N$, denote by $\Omega_1^l =\{ \omega \in   \Omega_1: s(\omega)= l \}$.
\end{definition}

\begin{lemma}\label{prop_tree_1}
The {odometric binary tree} verifies:
\begin{enumerate}
\item Every element $\omega\in \Omega_1$ appears exactly once as a leaf in the odometric binary tree.
\item For $l\geq 1$, we have $\mathcal{T}_1^l= \Omega_1^l$.
\end{enumerate}
\end{lemma}

\begin{proof}
We first prove (b). Note that, for  $l\geq 1$, the set $\Omega_1^l =\{ \bar\omega \in   \Omega_1: s(\bar\omega)= l \} $  has cardinality $2^{l-1}$. Indeed, any element with sum $l$ can be obtained from a list $1 \ 1 \dots 1 \ 1$ of length $l$, by intercalating  a $"+"$ symbol or a colon $","$ symbol at each  of the $(l-1)$ places between the ones, and performing the corresponding sums. Due to the sons rule, given $\omega \in \mathcal{T}_1^l$, we have $s(\omega^1)=s(\omega^{+1})=s(\omega)+1$. Hence, $\mathcal{T}^l_1\subset \Omega_1^l$. Since the tree is binary, and $\mathcal{T}_1^1$ has $1$ element, the cardinality of $\mathcal{T}_1^l$ is $2^{l-1}$, which implies $\mathcal{T}_1^l=\Omega_1^l$, concluding (b). For $\omega \in \Omega_1$, part (b) implies that $\omega \in \mathcal{T}_1^{s(\omega)}$. This fact, together with a cardinality argument, allow us to obtain (a).
\end{proof}

\begin{definition}
The {\it reverse lexicographic order} $<_{\Omega_1}$ on $\Omega_1$ is defined by:
\begin{enumerate}
\item For $\omega, \bar \omega\in \Omega_1$, if $s(\omega)<s(\bar \omega)$ then $\omega<_{\Omega_1}{\bar \omega}$.
\item If $s(w_1, w_2, \dots, w_j)=s(\bar w_1, \bar w_2, \dots, \bar w_ r)$ then  $\omega<_{\Omega_1}{\bar \omega}$ provided that
\[
w_j=\bar w_r, \ w_{j-1}=\bar w_{r-1}\ , \dots , w_{j-t-1}=\bar w_{r-t-1}, \ \textrm{and } w_{j-t}<\bar w_{r-t},
\]
for some $t \in \N$ with $j>t$ and $r>t$.
\end{enumerate}
\end{definition}

Note that this is well defined, the strict inequality in (b) occurs, otherwise both sums should be different.  An example of the order is $(1, 1, 3)<_{\Omega_1}(4, 2)$ and $(2, 2, 2, 1)<_{\Omega_1}(4, 2, 1)$. Note that the reverse lexicographic order $<_{\Omega_1}$ is a total order on $\Omega_1$.  
 
Given an integer $k\geq 0$, we define the block $\hat b_k=0 \underbrace{1 \dots 1}_{k-\textrm{times}}$. Note that $\hat b_k$ has length $k+1$ and that is the complementary version of the block $b_k$, defined in \eqref{bk}. To every $\omega \in \Omega_1$ we associate a positive integer $n_{\omega}$, which in binary notation corresponds to:
\[
n_{\omega}=\hat b_{w_j-1}\hat b_{w_{j-1}-1} \dots \hat b_{w_1-1}.
\] 
For instance, to  $\omega=(4, 2, 1)$ we associate  $n_{\omega}=0\ 01 \ 0111$, which in decimal notation is $n_{\omega}=23$. Note that $n_{(s)}=2^{s-1}-1$ and $n_{(1, \dots, 1)}=0$.

\begin{lemma}
Let $\omega , \bar\omega \in \Omega_1$. We have that, $\omega<_{\Omega_1}\bar \omega$ if and only if $2^{s(\omega)-1}+n_{\omega}<2^{s(\bar \omega)-1}+n_{\bar \omega}$.
\end{lemma}

\begin{proof}
Note that $n_{\omega}\in \{0, 1, \dots, 2^{s(\omega)-1}-1\}$, which implies $2^{s(\omega)-1}+n_{\omega}<2^{r}$ for any $r\geq s(\omega)$. When $s(\omega)=s(\bar \omega)$, the definitions allow us to conclude.
\end{proof}

We can now produce a more refined version of Lemma \ref{prop_tree_1}.

\begin{proposition}\label{prop_tree_2}
The odometric binary tree verifies:
\begin{enumerate}
\item The leafs in the level $\mathcal{T}_1^l$ appears from left-rightin the order induced by $<_{\Omega_1}$ on $\Omega_1^l$. In particular, the leftmost leaf is $(\underbrace{1,\dots, 1}_{l-\textrm{times}})$ and the rightmost leaf is $(l)$. 
\item The finite word $\omega=(w_1, w_2, \dots, w_j)$ corresponds to the leaf in the $n_{\omega}$ position of the level $\mathcal{T}_1^{s(\omega)}$. 
\end{enumerate}
\end{proposition}

\begin{proof}
We proceed by induction on $l$. Assume that (a) holds for $l\geq 1$. Take $\omega=(w_1, \dots, w_j), \bar \omega=(\bar w_1, \dots, \bar w_r)$, such that $\omega<_{\Omega_1}\bar \omega$ and $s(\omega)=s(\bar\omega)=l$ (that is, $\omega,\bar \omega\in \mathcal{T}_1^l$). It is easy to see that $\omega^{1}<_{\Omega_1}\omega^{+1}$. Since $s(\omega)=s(\bar\omega)$ there exists $t>1$ such that $w_t<\bar w_t$, that leads to $\omega^{+1}<_{\Omega_1}\bar \omega^{1}$, concluding the inductive step, and hence (a). Part (b) follows from (a) and Lemma \ref{prop_tree_1}.
\end{proof}

 
%
%
%

\begin{remark}
Note that Proposition \ref{prop_tree_2}.(b) implies  that the odometric binary tree is a {\it Cartesian Tree}, as introduced by Vuillemin \cite{VUI}. 
\end{remark}

\subsection{Odometric action on finite words.} In this section we extend the action of the odometer $O_1$ to finite words. There is a natural way to define $O_1$ for words of length strictly greater than one in a way that the action is consistent with the order induced by  $<_{\Omega_1}$.  For words of length one there is a great deal of freedom. We will concentrate on two possible choices that will be useful later.  
%
%

\begin{definition} \label{odo_finite}
Let $(w_1, w_2, \dots , w_j)  \in \Omega_1$. If  $j>1$ we define the action of $O_1$ by
\begin{equation}\label{def_odo}
O_1(w_1, w_2, \dots , w_j) =(\underbrace{1, \dots, 1,}_{w_1-1 \text{ times}} w_2+1, w_3, \dots, w_j).
\end{equation}
If $j=1$ the action $O_1(w_1)$ can be defined arbitrarily. That is, it can be any finite word.  The two main examples we consider are:
\[
O_1(w_1)=(\underbrace{1, 1, \dots, 1}_{w_1-\textrm{times}}) \quad \text{and} \quad O_1(w_1)=(\underbrace{1, 1, \dots, 1}_{(w_1+1)-\textrm{times}}).
\]
We call the first example  {\it cyclic odometer} while the  former  {\it top-down odometer}.
\end{definition}


\begin{lemma} \label{lema_1_count}
Let $\omega=(w_1, w_2, \dots, w_j)$ be a finite word, with $j>1$. Then, when regarded as leafs of the odometric binary tree, the leaf $O_1(\omega)$ is the leaf located one position to the right of $\omega$ at the same level.
\end{lemma}

\begin{proof} From equation \eqref{def_odo} we have that $s(O_{1}(\omega))=s(\omega)$. We will proceed by induction on $l=s(\omega)$. For $l=2$ we have $O_1(11)=(2)$. Assume the claim holds for $l\geq 2$. From the inductive hypothesis, $O_1(\omega)$ is the next leaf to the right from $\omega$ at the level $l$.  From equation \eqref{def_odo} we see that $O_1(\omega^1)=\omega^{+1}$. Therefore,
\begin{eqnarray*}
O_1(\omega^{+1})=O_1((w_1+1), w_2, \dots, w_j)&=&(\underbrace{1, \dots, 1}_{w_1-\textrm{times}}, w_2+1, \dots, w_j)\\
&=&(1, \underbrace{1, \dots, 1}_{(w_1-1)-\textrm{times}}, w_2+1, \dots, w_j)\\
&=&O_1(\omega)^1, 
\end{eqnarray*}
concluding the proof.
\end{proof}


\begin{lemma} \label{lema_2_count}
Let $\omega=(w_1, w_2, \dots, w_j)$ be a finite word, with $j>1$. Then,
\begin{equation}\label{n_odo}
n_{O_1(\omega)}=n_{\omega}+1.
\end{equation}
That is, $O_1(\omega)$ is the next finite word for the complete order $<_{\Omega_1}$, provided that the leaf $\omega$ is not a rightmost leaf in the tree.
\end{lemma}

The following result, which describes the orbits of the top-down odometer on finite words,   is a consequence Lemmas \eqref{lema_1_count} and \eqref{lema_2_count}.

\begin{theorem}[\bf Counting Finite Words]  \label{thm_counting_words} For every $w_1\geq 1$ define 
\[
O_1(w_1)=(\underbrace{1, \dots, 1}_{(w_1+1)-\textrm{times}}).
\]
Then, the map $n \mapsto O^n(1)$ is a bijection between $\N \cup\{0\}$ and $\Omega_1$. That is, the orbit of $(1)$ by $O_1$ passes trough every finite word in $\Omega_1$ once and only once. Moreover, the orbit follows the $<_{\Omega_1}$ order of $\Omega_1$. 
\end{theorem}


The next result is a  manifestation of the fact that odometers are {\it $+1$ adding machines}.

 \begin{corollary} Let $O_1$ be defined as in Theorem \ref{thm_counting_words}, then
 \[
 2^{s(O_1(\omega))-1}+n_{O_1(\omega)}=2^{s(\omega)-1}+n_{\omega}+1.
 \]
 \end{corollary}


 \subsection{Odometric tree and odometric action for $\Omega_k$} The results for $\mathcal{T}_1$ and $O_1$  can be extended so as to consider finite words having letters larger than an arbitrary number. Let $k\geq 0$, the set of finite words with symbols starting with $k$ will be denoted by  $\Omega_k=\bigcup_{j\geq 1}\N_k^j$. 
 
 \begin{definition}
Denote by $(\cdot)_k:\Omega_k\to\Omega_1$ the bijection between $\Omega_k$ and $\Omega_1$ defined  
for $\omega=(w_1, \dots, w_j)\in \Omega_k$  by
  \[
 (\omega)_k=(w_1, \dots, w_j)_k=(w_1-(k-1), \dots, w_j-(k-1)).
 \]
 The  {\it $k$-odometric binary tree}, denoted by $\mathcal{T}_k$, is obtained from $\mathcal{T}_1$ by applying $(\cdot)_k^{-1}$ to every leaf of $\mathcal{T}_1$.   The son-rule for $\mathcal{T}_k$ can be expressed by: a leaf $\omega=(w_1, w_2, \dots, w_j) \in \Omega_k$ has two sons
 \[
 \omega^1=(k, w_1, w_2, \dots, w_j)\quad \textrm{and}\quad \omega^{+1}=(w_1+1, w_2, \dots, w_j).
 \]
 \end{definition}

%
%
%
 
 \begin{definition}
The $k$-sum $s_k$ of a finite word $\omega=(w_1, \dots, w_j)\in\Omega_k$ is defined by:
 \[
 s_k(\omega)=s((\omega)_k)=\sum_{1\geq r\geq j} w_r-j(k-1).
 \]
The $k$-reverse lexicographic order on $\Omega_k$, denoted by $<_{\Omega_k}$, is defined by:
 \[
 \omega <_{\Omega_k}\bar \omega \quad\textrm{if and only if} \quad (\omega)_k<_{\Omega_1}(\bar\omega)_k.
 \]
 Let $\Omega_k^{l}=\{\omega\in \Omega_k\ | \ s_k(\omega)=l\}$.
    \end{definition}

 \begin{proposition}\label{prop_k_tree_2}
 The $k$-odometric binary tree verifies:
\begin{enumerate}
\item The leafs in the level $\mathcal{T}_k^l$ appears from left-right in the order induced by $<_{\Omega_k}$ on $\Omega_k^l$. In particular, the leftmost leaf is $(\underbrace{k,\dots, k}_{l-\textrm{times}})$ and the rightmost leaf is $(l)$. 
\item The finite word $\omega=(w_1, w_2, \dots, w_j)$ is located at the leaf in the $n_{(\omega)_k}$ position of the level $\mathcal{T}_k^{s_k(\omega)}$. 
\end{enumerate}
\end{proposition}

\begin{definition}
We define the $k$-odometric action  $O_k:\Omega_k\to \Omega_k$ by $O_k(\omega)=(O_1((\omega)_k))_k^{-1}$. In other words, for $(w_1,w_2,  \dots, w_j) \in \Omega_k$, with $j>1$ we have
\[
O_k(w_1, \dots, w_j)=(\underbrace{k, \dots, k, }_{(w_1-k)-\textrm{times}} (w_2+1), \dots, w_j).
\]
As in the case for $k=1$, the action $O_k$ at words of the form $(l)$ can be defined in an arbitrary fashion.
\end{definition}

\begin{lemma}
If $\omega=(w_1, \dots, w_j)\in \Omega_k$, with $j>1$, then  $O_k(\omega)$ is the next finite word for the complete order $<_{\Omega_k}$. Hence, provided that $\omega$ is not a rightmost leaf in $\mathcal{T}_k$,  then $O_k(\omega)$ is the next leaf to the right of $\omega$, at the level $\mathcal{T}_k^{s_k(\omega)}$. 
\end{lemma}

\subsection{Examples}\label{examples} There exists different numeration systems with countably many symbols. Finite words can be interpreted using them. For example,  the continued fraction and the backward continued fraction systems. In these cases, finite words correspond to rational numbers. By means of these identifications we will construct, or recover, both the Kepler and the Calkin and Wilf trees. Moreover, the odometric action on these trees allow us to recover different ways of counting rational numbers in a unified way.  We begin, though, with the more standard binary expansion and its relation to the usual dyadic odometer.

\subsubsection{Binary expansion and Dyadic Tree.} \label{dyd} Dyadic numbers in $(0,1)\subset \mathbb{R}$ are rationals that can be written as a finite binary expansion $x=\sum_{1\leq n \leq j}\frac{a_n}{2^n}$, with $a_n\in\{0,1\}$ and  $a_j=1$. Therefore, to every dyadic number $x \in (0,1)$ we can associate a finite word $(a_1, \dots, a_j)$ with $a_i \in \{0,1\} $. As in the proof of Theorem \ref{main_odo}, we can recode this finite word in $\Omega_0$  using the blocks $b_m=\underbrace{1, \dots, 1}_{m-\textrm{times}},0$. The final block in the form $\underbrace{1, \dots, 1}_{m-\textrm{times}}$ will be coded by $b_m$. For example, $\frac{19}{32}=0,10011\mapsto (1, 0, 2)\in \Omega_0$. 
The $0$-odometric binary tree $\mathcal{T}_0$ has finite words ending by $0$. This produces repetitions in the coding of dyadic numbers. Nevertheless, the sub-tree $\mathcal{T}_{0,(1)}$, composed by the descendants of $(1)$ never ends by $0$, and contains a unique coded version of every dyadic number in $(0,1)\subset \R$. The leftmost leaf of  the level $l$ is $(\underbrace{0, \dots, 0 }_{(l-1)-\textrm{times}}, 1)$ and the rightmost leaf is $(l)$.  The following are the son-rules for $\mathcal{T}_{0,(1)}$ in terms of dyadic numbers:
the number $x=0,b_{w_1}b_{w_2}\dots b_{w_j}$ has two sons, in the following order
\[
0,b_0b_{w_1}b_{w_2}\dots b_{w_j}=\frac{x}{2} \quad \textrm{and} \quad 0,b_{w_1+1}b_{w_2}\dots b_{w_j}=\frac{1+x}{2}.
\]

\begin{figure}
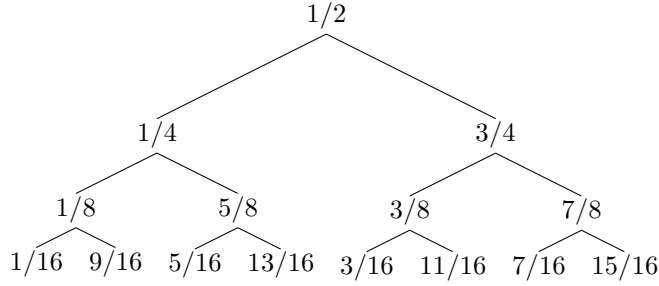

 \Tree [.1/2 [.1/4 [.1/8 [.1/16 ] [.9/16 ] ] [.5/8 [.5/16 ] [.13/16 ] ] ] [.3/4 [.3/8 [.3/16 ] [.11/16 ] ] [.7/8 [.7/16 ] [.15/16  ] ] ]]
\centering
 \caption{The first four levels of the dyadic tree or the  $\mathcal{T}_{0,(1)}$ tree.}
  \label{dyadic_tree}
\end{figure}

See Figure \ref{dyadic_tree} for the first four levels of this tree. The odometric action $O_0$ defined so that $O_0(w_1)=(\underbrace{0, \dots, 0 }_{w_1-\textrm{times}}, 1)$ (that is, the top-down definition restricted to $\mathcal{T}_{0, (1)}$), coincides with the action of the  the von Newmann Kakutani odometer restricted to dyadic numbers (see Section \ref{odometro_en_I}). 

\begin{remark} \label{count_dyadic}
It is a direct consequence of Theorem \ref{thm_counting_words}  that the map $n \mapsto O_0^n(1)$ is a bijection between the sets of finite words with symbols in $\{0,1\}$  and   $\N \cup\{0\}$.
\end{remark}

\subsubsection{Continued fractions and the Kepler Tree.} \label{con_f} Every real number $ x \in (0,1)$  can be written as a continued fraction of the form
\begin{equation*}
x = \textrm{ } \cfrac{1}{a_1 + \cfrac{1}{a_2 + \cfrac{1}{a_3 + \dots}}} = \textrm{ } [a_1\ a_2\ a_3\ \dots],
\end{equation*}
where $a_i \in \mathbb{N}$. As it is well known \cite[Theorem 170]{hw}, arithmetic properties of the number are captured by the expansion. In particular, a number $x \in (0,1)$ is irrational if and only if its continued fraction expansion  has infinitely many terms. Moreover, irrational numbers have a unique continued fraction expansion. On the other hand,  rational numbers have two different expansions, both of them finite and one of them having last digit equal to $1$. In fact $[a_1\ a_2\ a_3\ \dots\ a_n \ 1]=[a_1\ a_2\ a_3\ \dots\ (a_n+1)]$. Hence, every rational in $(0,1)$ has two representatives as leafs of the tree $\mathcal{T}_1$. If we restrict ourselves to the sub-tree of the descendants of $(2)$, denoted by $\mathcal{T}_{1, (2)}$, we obtain that every rational number in $(0,1)$ appears exactly once, see Figure \ref{kepler_tree}.  The son-rule can be computed in terms of continued fractions. Indeed, if $\frac{p}{q}=[a_1\ a_2\ \dots \ a_n]$ then
\[
[1\ a_1 \ \dots \ a_n]=\cfrac{1}{1+\cfrac{1}{a_1 + \cfrac{1}{a_2 + \dots}}}=\cfrac{1}{1+\cfrac{p}{q}}=\frac{q}{p+q}, 
\]
 and
 \[
[ (a_1+1) \ \dots \ a_n]=\cfrac{1}{(a_1 +1) + \cfrac{1}{a_2 + \dots}}=\cfrac{1}{1+\cfrac{q}{p}}=\frac{p}{p+q}.
\]

\begin{figure}
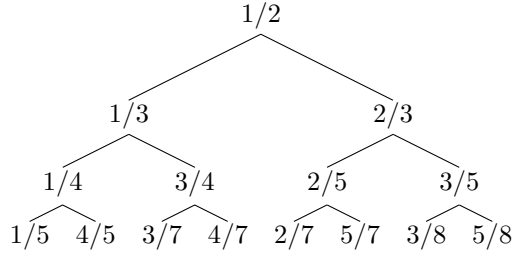
 
 \Tree [.1/2 [.1/3 [.1/4 [.1/5 ] [.4/5 ] ] [.3/4 [.3/7 ] [.4/7 ] ] ] [.2/3 [.2/5 [.2/7 ] [.5/7 ] ] [.3/5 [.3/8 ] [.5/8  ] ] ]]
\centering
 \caption{The first four levels of the  Kepler tree or the $\mathcal{T}_{1, (2)}$ tree.}
  \label{kepler_tree}
\end{figure}

Hence, the tree ${\mathcal{T}_{1, (2)}}$ coincides with the well known {\it Kepler tree} (see \cite{kep, oeis} or \cite[p.383]{b}). In particular, every rational number in $(0,1)$ appears once in this tree, and the exact position in the tree can be computed from its continued fraction. 

\begin{remark}
The top-down odometric action on the Kepler tree induces a bijection between the rational numbers in $(0,1)$ and $\N \cup \{0\}$. Indeed, it suffices to consider the map $n \mapsto O_1^n((2))$ and to identify each finite word $O_1^n((2))$ with the corresponding rational.
\end{remark}

In section \ref{cfe} we give a complete account of the dynamics of the odometric action $O_1$ on the interval and, in particular, on the Kepler tree.

%
%
%
%

\subsubsection{Backward continued fractions and the Calkin and Wilf tree.} \label{con_b} Every real number $x \in [0,1)$ can be written as backward continued fraction of the form
\begin{equation*}
x = 1- \textrm{ } \cfrac{1}{a_1 - \cfrac{1}{a_2 - \cfrac{1}{a_3 - \dots}}} = \textrm{ } [a_1\ a_2\ a_3\ \dots]_B,
\end{equation*}
where  $a_i \geq 2$. Irrational numbers  have a unique  backward continued fraction expansion which is infinite. Rational numbers  have two different expansions: a finite one  and an infinite one of  that ends with a tail of $2$'s (see \cite[Theorem 1.2]{ka}). In fact, 
\begin{equation}\label{finite_2_B}
[a_1\ a_2\ a_3\ \dots \ a_n]_B=[a_1\ a_2\ a_3\ \dots \ (a_n+1)\ 2\ 2\ 2\ \dots]_B.
\end{equation}
Hence, every rational in $(0,1)$ has a unique representative as a leaf of the tree $\mathcal{T}_2$. The son rule can be computed in terms of the backward continued fraction. If $\frac{p}{q}=[a_1\ a_2\ a_3\ \dots \ a_n]_B$ then
\[
[2\ a_1 \ \dots \ a_n]_B=1-\cfrac{1}{2-\cfrac{1}{a_1 - \cfrac{1}{a_2 - \dots}}}=1-\cfrac{1}{1+\cfrac{p}{q}}=\frac{q}{p+q}, 
\]
and
\[
[ (a_1+1) \ \dots \ a_n]_B=1-\cfrac{1}{(a_1+1) - \cfrac{1}{a_2 - \dots}}=1-\cfrac{1}{1+\cfrac{q}{q-p}}=\frac{q}{2q-p}.
\]
The Calkin and Wilf tree, see \cite[p.361]{cw}, has the following son-rule: every leaf $\frac{a}{b}$ has two sons, $\frac{a}{a+b}$ and $\frac{a+b}{b}$. Thus, leafs alternate  rationals in $(0,1)$ and rationals larger than $1$. The odd terms (that is, the leafs that are in $(0,1)\subset \mathbb{R}$) form a sub-tree with the following son-rule: the leaf $\frac{a}{a+b}$ has two sons that appear in the following order: $\frac{1}{2a+b}$ and $\frac{a+b}{a+2b}$, see Figure \ref{calkin_tree}. 


%
%
%
%
%

\begin{figure}
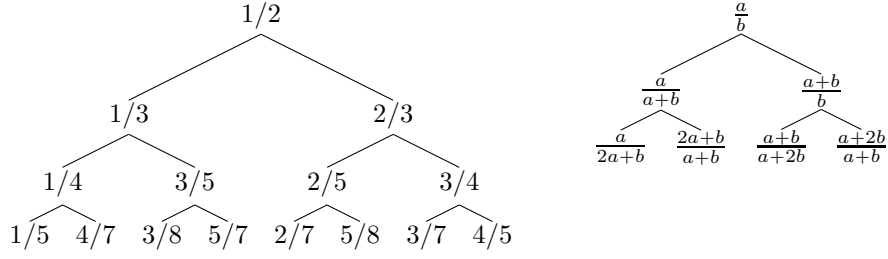

  \Tree [.1/2 [.1/3 [.1/4 [.1/5 ] [.4/7 ] ] [.3/5 [.3/8 ] [.5/7 ] ] ] [.2/3 [.2/5 [.2/7 ] [.5/8 ] ] [.3/4 [.3/7 ] [.4/5  ] ] ]]
  \hskip 0.3in
  \Tree [.$\frac{a}{b}$ [.$\frac{a}{a+b}$ [.$\frac{a}{2a+b}$ ] [.$\frac{2a+b}{a+b}$ ] ] [.$\frac{a+b}{b}$ [.$\frac{a+b}{a+2b}$ ] [.$\frac{a+2b}{a+b}$ ] ]]
\centering
 \caption{The first four levels of the Calkin and Wilf sub-tree or the $\mathcal{T}_{2, (2)}$ tree and the son-rule.}
  \label{calkin_tree}
\end{figure}

Replacing $a=p$ and $q=a+b$, we see that the tree $\mathcal{T}_{2, (2)}$ coincides with the sub tree of the Calkin and Wilf tree of rationals in $(0,1)$ (which, in turn, coincides with the odd leafs at every level), see Figure \ref{calkin_tree}. In particular, every rational in $(0,1)\subset \mathbb{R}$ appears once and only once in this tree. Moreover, its exact position 
can be computed from the backward continued fraction. Since every rational has an expansion with infinitely many digits, the odometric action $O_2$ has a natural definition on this tree. It turns out that the odometric action runs through every leaf from left-right to top-down. More details in section \ref{sbcf}.

\begin{remark} \label{count_rat}
It is a direct consequence of Theorem \ref{thm_counting_words} and that $O_2$ is uniquely defined for rational numbers (by means of their infinite expansion) that the map $n \mapsto O_2^n(1)$ is a bijection between the sets of rational numbers in $(0,1)$ and   $\N \cup\{0\}$. Indeed, we only need to identify $O_2^n(1)$ with the corresponding rational.
\end{remark}

\section{Odometers in non-compact  Polish spaces}

Using universality properties of the Baire space it is possible to extend the results in section \ref{sec:def} to a broad range of spaces. Recall that a separable, completely metrizable space is called \emph{Polish space}. A topological space is \emph{zero dimensional} if it is Hausdorff and has a basis consisting of clopen sets. The odometer defied in the non-compact space $\Sigma_{\geq 1}$ can be used to define odometers in several spaces. 


\begin{theorem} \label{au}
Let $X$ be nonempty Polish zero-dimensional space, for which all compact subsets have empty interior.
 Then, there exists a continuous map $O_{X}: X \to X$ topologically conjugated to $O_1:  \Sigma_{\geq 1}  \to \Sigma_{\geq 1}$.
\end{theorem}

\begin{proof}
Alexandrov and Urysohn \cite[Theorem 7.7 p.37]{ke} proved that the Baire space  is the unique, up to homeomorphism, nonempty Polish zero-dimensional space, for which all compact subsets have empty interior.
Therefore, there exists an homemorphism $g: \Sigma_{\geq 1} \to  X$.
It suffices to define $O_{X}: X \to X$ by
\begin{equation*}
O_{X}=   g  \circ O_0 \circ g^{-1}.
\end{equation*}
This readily yields the desired transformation. 
\end{proof}

This result can be applied in the setting of countable Markov shifts. Let $M$ be a $\N\times \N$ matrix  with entries $0$ or $1$. The \emph{symbolic space} associated to $M$ with alphabet $\N$ is defined by
 \begin{equation*}
 \Sigma_M:=\left\{ (x_1, x_2, \dots) \in \Sigma_{\geq 1}: M(x_i, x_{i+1})=1 \text{ for every } i \in \N \right\},
\end{equation*} 
As before, we endow $\N$ with the discrete topology and $\N^{\N}$ with the product topology. On $\Sigma_M$ we consider the induced topology given by the natural inclusion $\Sigma\subset \N^{\N}$. We stress that, in general, this is a non-compact space.  The space $\Sigma_M$ is locally compact if and only if for every $i \in \N$  we have $\sum_{j \in \N} M(i,j ) <\infty$ (see \cite[Observation 7.2.3]{ki}). The space $\Sigma_M$ is metrizable. The following result is direct consequence of Theorem \ref{au}.

\begin{corollary} \label{symb}
Let $\Sigma$ be a non-locally compact  symbolic space, then there exists a continuous map $O_{\Sigma}: \Sigma \to \Sigma$ topologically conjugated to $O_1:  \Sigma_{\geq 1}  \to \Sigma_{\geq 1}$.
\end{corollary}

%
%

Moore generally,  for any Polish space with no isolated points we can induce, by semi-conjugacy, an odometer. Indeed, it suffices  to proceed as in Theorem \ref{au} noticing that Sierpinski  \cite[7.15 p.40]{ke}    proved that if $X$ is a non empty Polish space with no isolated points then there exists a continuous bijection $g: \Sigma_{\geq 1} \to X$.  We stress that, in this case, the map $g^{-1}$ need not to be continuous.
%
%
%

\section{Odometers in the interval} \label{sec:i}

Numeration systems have been related to odometers, at least, since the work of von Neumann and that of Kakutani. In this section, by means of the results in section \ref{sec:def}, we extend this relation to a wide range of numeration  systems using infinite alphabets.   Before moving to the countable alphabet setting we recall the interval map realization of the odometer $O$.  Let  $(I_n)_n$ be the partition of the unit interval, where $I_n=[\frac{2^{n-1}-1}{2^{n-1}}, \frac{2^{n}-1}{2^{n}})$, for $n \in \N$. The 
interval realization of the \emph{dyadic odometer} is the interval map $\tilde{O}: [0,1) \to [0,1)$ defined by
\begin{equation}\label{odometro_en_I}
\tilde{O}(x) = x +\frac{3}{2^{n}} -1, 
\end{equation}
for $x \in I_n$. We have that $\tilde{O}: [0, 1] \setminus \{{\text{dyadic rationals}} \} \to  [0, 1] \setminus  \{{\text{dyadic rationals}}\}$ is topologically conjugated to the odometer $O$ defined in the space $\Sigma_2 $ minus the sequences that end with a tail of ones or with a tail of zeros. Moreover, the action of $\tilde{O}$ can be extended to the dyadic rationals. Thus, we obtain an interval map realization of Remark \ref{count_dyadic}.

\begin{remark}
The map $n \mapsto \tilde{O}(1/2)$ is a bijection between the sets of dyadic numbers in $(0,1)$ and $\N \cup \{0\}$. That is, the orbit of $1/2$ by $\tilde{O}$ passes through every dyadic number in $(0,1)$ once and only once.  
\end{remark}

\subsection{Odometers obtained from interval maps}

Given an interval map with countably many  full branches, by means of the results in section \ref{sec:def}, we construct an associated odometer. Denote by $\text{dom}(T)$ the domain of the function $T$. We now define the class of interval maps we consider.   

\begin{definition}
Let $(I_n)_n$ be a collection of intervals in $[0,1]$ having disjoint interiors. A map $T: \bigcup_{n \geq 1}I_n \mapsto [0,1]$ is a Countable-Markov-Interval map (CMI map) if the following holds:
\begin{enumerate}
\item The map $T$ is full-branched, that is, for every $n \in \N$ we have $(0,1) \subset T(I_n)$.
\item There exists $\alpha >0$ such that the map is pieecewise of class $C^{1+ \alpha}$.
\item For every $x \in (0,1)$ we have that $|T'(x)|>1$.
\item If  $x \in \text{dom}(T) \cap \{0,1\}$ then  $|T'(x)| \geq 1$.\end{enumerate}
\end{definition}

For a CMI-map $T$ its corresponding \emph{repeller} is the set
\begin{equation*}
\Lambda_T= \left\{x \in \text{dom}(T): T^n(x) \text{ is well defined for every } n \in \N \right\}.
\end{equation*}
The set of points with \emph{finite orbit} for $T$ is defined by 
$\mathcal{F}= \text{dom}(T) \setminus \Lambda_T$. 
We stress that for some maps  this set may be empty. For an interval $I$ let $\partial I$ be the boundary points of $I$. Denote the set of \emph{boundary} points of $T$ by:
\begin{equation*}
\mathcal{B}_T=  \bigcup_{n \geq 0} T^{-n}\left(\{0,1\} \cup \bigcup_n \partial I_n	\right).
\end{equation*}
The following holds, 

\begin{lemma}
For every CMI-map $T$ there exist a continuous  \emph{coding map} $\pi_T : \Sigma_{\geq 1} \to \Lambda_T$ such that
\begin{equation*}
\pi_T \circ T =  \sigma \circ \pi_T.
\end{equation*}
\end{lemma}

\begin{proof}
The proof of this results is well known, see for example \cite[Proposition 1.2]{sa}, for points not belonging to the boundary of the partition $(0,1) \setminus \mathcal{B}_T$. Since all  iterates of the map $T$ are well defined in $\Lambda_T$, the map $\pi_T$ can actually be defined be defined in $\text{dom}(T) \setminus \mathcal{F}= \Lambda_T$.
\end{proof}

\begin{remark} \label{homeo}
The map $\pi_T :  \Sigma_{\geq 1} \setminus \pi_T^{-1} \left( \mathcal{B}_T \right) \to \Lambda_T \setminus \mathcal{B}_T$
is an homemorphism.
\end{remark}

\begin{remark}
Note that we do not require that $(0,1) \subset \bigcup_n I_n$. Also, the closure of the points in the boundary of the partition can be very large. It can be, for example,  a Cantor set.
\end{remark}

To every CMI-map $T$ we can associate an odometer in the following way.

\begin{definition} \label{odo_cmi}
Let $O_T: \Lambda_T \to [0,1]$ be the map defined by,
\begin{equation*}
O_T(x)=   \left( \pi_T  \circ O_1 \circ \pi_T^{-1} \right) (x).
\end{equation*} 
\end{definition}

It follows from the definition of $O_T$,  Remark \ref{homeo}  and Theorem \ref{main_odo} that, away from the boundary points,  the map $O_T$ is topologically conjugated to the odometer in a countable alphabet. If $\Lambda_T=[0,1)$ or $\Lambda_T=(0,1]$ the map $O_T$ can be thought of as an infinite interval exchange transformation.

\begin{theorem}
Let $T$ be an CMI-map and $O_T$ the associated odometer.  Then,  the map $O_T: \text{dom}(T) \setminus  \mathcal{B}_T: \to [0,1]$ is topologically conjugated to the $O_0$ odometer.
\end{theorem}

\begin{corollary}
The map $O_T: \text{dom}(T) \setminus  \mathcal{B}_T: \to [0,1]$   satisfies the following dynamical properties:
\begin{enumerate}
\item The map $O_T$  is minimal and uniquely ergodic.
\item The map $O_T$ has zero topological entropy.
\item The map $O_T$  has purely discrete spectrum and it is equicontinuous.
\end{enumerate}
\end{corollary}

The behavior of $O_T$ on points with finite orbits $\mathcal{F}$ will depend on how we define $T$ on the boundary points $\partial I$. In general, we are interested in continuous extensions of $T$ at least from one-side at each point in $\partial I$.  In what follows we provide explicit examples of maps $O_T$. 
%
%
%
%
%
%

%
%

\subsection{The odometer associated to the continued fraction expansion} \label{cfe}

Recall that, see section \ref{con_f}, every number in $(0,1)$ can be written as a continued fraction expansion. There exists CMI-map closely related to this expansion. Indeed, the \emph{Gauss} map  $G :(0,1] \to (0,1]$, is the interval map defined by 
\begin{equation*}
G(x)= \frac{1}{x} -\left[ \frac{1}{x} \right], 
\end{equation*}
where $[\cdot]$ is the integer part. If $x=[a_1 a_2 a_3 \dots]$ then $G(x)=[ a_2 a_3 \dots]$.  The repeller for $G$ is $\Lambda_G=[0,1]\setminus \Q$ and the set with finite orbits is $\mathcal{F}=\Q$. We will assume that the continued fraction expansion of rationals does no ends with a one. Denote by $O_G:(0,1] \to (0,1]$ the transformation obtained from $G$ as given in Definition \ref{odo_cmi} .  The action of $O_G$ over the irrational numbers is, by definition, given by:
\begin{lemma} 
 If $x=[a_1, a_2, a_3, \dots] \in [0,1)$ then
 \begin{equation*}
 O_G(x)=[\underbrace{1,1,\dots,1}_{(a_1-1)\textrm{-times}}, a_2+ 1, a_3, a_4, \dots].
 \end{equation*}
\end{lemma}

We have two natural choices to define $O_G$ at the boundary points $[n]$. In the first case we define $O_G$ to be continuous from the right, that is
 \[
 O_G[n]=\lim_{m\to \infty}O_G[(n-1)\ 1 \ m]=\lim_{m\to \infty} [\underbrace{1\dots 1}_{n-2}\ 2 \ m]= [\underbrace{1\dots 1}_{n-2}\ 2 ].
 \]
In that case, $[n]$ is periodic for $O_G$, with period $2^{n-2}$. With this choice, every rational number $[a_1 a_2 \dots a_j]$ is periodic, with period $2^{\sum_i{a_i}-2}$. This corresponds to the cyclic odometer, see Definition \ref{odo_finite}.

On the other hand, we could define $O_G$ to be continuous from the left. In this case we obtain,
\[
 O_G[n]=\lim_{m\to \infty}O_G[n \ m]=\lim_{m\to \infty} [\underbrace{1\dots 1}_{n-1}\   (m+1)]= [\underbrace{1\dots 1}_{n-1} ].
 \] 
 With this choice, the orbit of every rational number  goes up the Kepler Tree until it reaches $[1]$.  Hence, every rational number in $(0,1)$  reaches $[1]$ in finite time. In what follows we will consider the first choice, that is, we assume $O_G$ to be continuous from the right.  Recall that the {\it Fibonacci Sequence} is defined by,
\begin{equation*}
f_{0}=0, \ f_1=1, \ f_{n+1}=f_n+f_{n-1}.
\end{equation*}

The following result explicitly describes the odometer $O_G$ (when it is continuous from the right).

\begin{proposition} \label{action-gauss}
For $x\in \left[\frac{1}{n+1}, \frac{1}{n}\right) \cap \Q^c$ we have that
\begin{equation*}\label{fibo_odo}
O_{G}(x)=\frac{x(f_{n-1}-nf_{n})+f_{n}}{x(f_n-nf_{n+1})+f_{n+1}}.
\end{equation*}
If $x=[a_1, \dots , a_n] \in \Q$ then the point $x$ is periodic for $O_G$ with period $2^{-2+\sum_{i=1}^n a_i}$.
\end{proposition}

\begin{proof} 
Any $x\in \left[\frac{1}{n+1}, \frac{1}{n}\right)$ can be uniquely written as
\begin{equation}\label{x_a}
x=[n\ a]=\frac{1}{n+\frac{1}{a}},
\end{equation}
with  $a\geq 1$ a real number. The action of the odometer is
\begin{equation}\label{fibo_a}
O_{G}[n\ a]=[\underbrace{1\dots 1 }_{n-1}\ (a+1)].
\end{equation}
We claim that (\ref{fibo_a}) leads to
\begin{equation}\label{odo_a}
O_{G}[n\ a]=\frac{f_{n-1}a+f_n}{f_{n}a+f_{n+1}}.
\end{equation}
Indeed, reasoning by induction on $n$, we have the base case
\begin{eqnarray*}
O_{G}[1\ a]=[a+1]=\frac{1}{a+1}.
\end{eqnarray*}
Consider now the case $(n+1)$, assuming that (\ref{odo_a}) holds
\begin{eqnarray*}
O_{G}[(n+1)\ a]&=&[\underbrace{1\dots 1}_{n} \ (a+1)]\\
&=&\frac{1}{1+[\underbrace{1\dots 1}_{n-1}\ (a+1)]}\\
&=&\frac{1}{1+\frac{f_{n-1}a+f_n}{f_{n}a+f_{n+1}}}\\
&=&\frac{f_{n}a+f_{n+1}}{a(f_{n}+f_{n-1})+f_{n+1}+f_{n}}=\frac{f_{n}a+f_{n+1}}{f_{n+1}a+f_{n+2}}.
\end{eqnarray*}

The equation (\ref{x_a}) leads to
\[
a=\frac{x}{1-nx},
\]
and then
\begin{eqnarray*}
O_{G}(x)&=&\frac{f_{n-1}\left(\frac{x}{1-nx}\right)+f_n}{f_{n}\left(\frac{x}{1-nx}\right)+f_{n+1}}\\
&=&\frac{x(f_{n-1}-nf_n)+f_n}{x(f_n-nf_{n+1})+f_{n+1}}.
\end{eqnarray*}
\end{proof}

The map $O_G$ has one fully supported ergodic measure which corresponds to the projection of the unique invariant measure for $O_1$. Moreover, since the rational points are  periodic, the map also preserves countably many ergodic measures supported on periodic orbits. There are no other ergodic invariant measures. We therefore have,

\begin{lemma}
The space of invariant probability measures for $O_G$ is a Choquet simplex with countably many extreme points.
\end{lemma}

\begin{remark}
The function $O_G$, or its inverse, has occurred previously in the literature. For example, in \cite{vpb}  Viader, Paradis and Bibiloni proved that the  question mark function, introduced by Minkowski, corresponds to the asymptotic distribution of  the  enumeration of the rational numbers in $(0,1)$ given by the Kepler tree, see section \ref{con_f}. Note, however, that they do not explicitly mention the Kepler tree and that their construction of the order is obtained in a different, although equivalent, way than the one induced by the ododmetric action. In \cite[Section 2.1]{vpb}, computing the successor in the order of the Kepler tree,  they obtain an explicit formula for the inverse of $O_G$. Also, Panti \cite{pa} constructed Kakutani von-Neumann odometers in higher dimensional simplexes. The one dimensional version of his construction is related to the inverse of  the map $O_G$.
\end{remark}

\begin{remark} 
We provide a flow version of the example described on this section that may be of interest.  Let $\tau:(0,1] \to \R$ be a continuous positive map bounded away from zero.  Denote by 
$Y=\left\{ (x,t)  \in (0,1) \times \R :  0 \leq t \leq \tau(x)	\right\}$. Consider the equivalence  relation $(x,\tau(x)) \sim (O_G(x), 0)$ and the quotient space $Y/ \sim$. The flow $\Phi=(\phi_t)_t$ defined on the quotient space by
\begin{equation*}
\phi_t(x,s)=(x, s+t) \quad \text{ whenever } \quad s+t \in [0, \tau(x)].
\end{equation*}
This flow satisfies the following properties: it has one, and only one, fully supported measure and countably many measures supported on periodic orbits. Every point not in a periodic orbit is generic for the fully supported measure.
A renormalization formula for the flow $\Phi$ and the the suspension flow with base the shift on countably many symbols and roof function $\tau$ can be obtain from the corresponding property established in section \ref{reno}.
\end{remark}

\subsection{The odometer associated to the backward continued fraction expansion}\label{sbcf}

Recall that, see section \ref{con_b}, every number in $(0,1)$ can be written as a backward continued fraction expansion. There exists CMI-map closely related to this expansion. Indeed, the {\it Renyi map}, $R:[0,1) \to [0,1)$, is defined by
\begin{equation*}
R(x)=\frac{1}{1-x} -\left[\frac{1}{1-x} \right].  
\end{equation*}
This map acts as the shift on the backward continued fraction (see \cite[Chapter 11]{dk}). That is, 
if $x =[a_1 a_2 a_3 \dots]_B$   then  $R(x)=[a_2 a_3 \dots]_B$. Note that $R(0)=0$ and that $R'(0)=1$. In particular, the Renyi map is not uniformly expanding.   Renyi \cite{re} showed that there exists an infinite $\sigma-$finite invariant measure, $\mu_R,$  absolutely continuous with respect to the Lebesgue measure. It is defined by
\[ \mu_R(A) = \int_A \dfrac{1}{x} \ dx, \]
where $A \subset [0,1)$ is a Borel set.  There is no finite  invariant measure absolutely continuous with respect to the Lebesgue measure.   Denote by $O_R:[0,1) \to [0,1)$ the transformation obtained from $R$,  as given in Definition \ref{odo_cmi}. In \cite{ip} we proved, using Minkowski's question mark function, that:

\begin{proposition} \label{bcf}
For every $x \in [0,1)$ we have that,
\begin{equation*}
O_R(x)= \frac{1}{2\left[\frac{1}{1-x}	\right] + 1 - \frac{1}{1-x}}.
\end{equation*}
\end{proposition}

Moreover, we established the following result which, in this setting, follows directly from the definition of $O_R$:
\begin{lemma} 
 If $x=[a_1, a_2, a_3, \dots]_B\in [0,1)$ then
 \begin{equation*}
 O_R(x)=[\underbrace{2,2,\dots,2}_{(a_1-2)\textrm{-times}}, a_2+ 1, a_3, a_4, \dots]_B.
 \end{equation*}
\end{lemma}

Interest in the map $O_R$ arises from the fact that the orbit of $x=0$ induces an explicit  bijection between the sets of natural numbers  and the rational numbers  in  $(0,1)$. Indeed, the map $n \mapsto O_R^n(0)$ is the required bijection. 
Of course, this readily follows Remark \ref{count_rat} and the general ideas developed in this article. However, this fact was first observed by Newman, while solving  a problem posed by Knuth \cite{k}, although the way $O_R$ is written is formally different. The proof original given  of this fact is based on work by Calkin and Wilf \cite{cw} and can be found in \cite[Chapter 17]{az}. They considered  the sequence $q(n)$ defined as the number of ways of writing the integer $n$ as a sum of powers of $2$, each power being used at most twice.  They go on to prove that the map $n \to b(n)/b(n+1)$ is a bijection between the natural and the rational numbers. It turns out that the numbers $q(n)$ appear  in the Calkin and Wilf tree (also known as the Euclid tree). Newman  constructed a map for which the orbit of zero realizes the Calkin and Wilf bijection. Bonanno  and  Isola \cite[Theorem 2.3]{bi} proved that the map $O_R$  is topologically conjugated to the interval map realization of  dyadic odometer. Their proof uses, instead of backward continued fractions, the usual continued fractions. Again, the formal expressions for $O_R$ are different. As we already mentioned, this counting property is a manifestation of Theorem \ref{thm_counting_words} and of its version described in Remark  \ref{count_rat}.

Since every finite sequence can be realized as in infinite sequence ending by a tail of $2$'s (see equation \eqref{finite_2_B}), for the map $R$  the set $\mathcal{F}$ is empty.  Or, if we prefer, the odometric action on finite words corresponds to the top-down odometer, see Definition \ref{odo_finite}. Therefore,

\begin{proposition}
The map $O_R:[0,1) \to [0,1)$ is uniquely ergodic.
\end{proposition}
It turns out that the only invariant measure is the projection of the Lebesgue measure through the question mark function (see \cite{bi, ip}). In particular, the invariant measure is singular with respect to the Lebesgue measure.

\subsection{The odometer on numbers with partial quotients larger than $k$}
In the previous examples the repeller together with the points with finite orbits form an interval. In this section we provide an example in which this is not the case.  Let $k \in \N$ be a positive integer larger or equal than $2$. Consider the set,
\begin{equation*}
\Lambda_k=\left\{x=[a_1,a_2, \dots]  \in (0,1) : a_i \geq k \text{ for every } i \in \N		\right\}.
\end{equation*}
This is a totally disconnected set having positive Hausdorff dimension strictly less than $1$. Let $G_k$ be the restriction of the Gauss map to the set $\Lambda_k$. Denote by $O_{G_k}:(0,1/k] \to (0,1]$ the transformation obtained from $G_k$,  as given in Definition \ref{odo_cmi}.  Consider the following two  $k-${\it Fibonacci Sequences} defined by,
\begin{eqnarray*}
b_{0}=0, \ b_1=1, \ b_{n+1}=k b_n+b_{n-1}, &\\
d_{0}=1, \ d_1=1, \ d_{n+1}=k d_n+d_{n-1}.
\end{eqnarray*}
An argument similar to that in Proposition \ref{action-gauss} leads to the following:

\begin{proposition} \label{action-k-gauss} 
Let $n \in \N \cup \{0\}$. If $x\in \left[\frac{1}{n+1+k}, \frac{1}{n+k}\right) \cap \Q^c \cap \Lambda_k$ then
\begin{equation*}\label{fibo_odo_k}
O_{G_k}(x)=\frac{x(b_{n}-nd_{n+1})+d_{n+1}}{x(b_{n+1}-nd_{n+2})+d_{n+2}}.
\end{equation*}
\end{proposition}

%
%

\begin{remark}
Note that for $n \in \N$, if $x \in (1/(n+1+k), 1/(n+k))$ then
\begin{equation*}
\lim_{x \to  \frac{1}{n+1+k}}O_{G_k}(x)= \frac{b_n}{b_{n+1}}.
\end{equation*}
Moreover,
\begin{equation*}
\lim_{n \to \infty} \frac{b_n}{b_{n+1}}= \frac{1}{\phi_k},
\end{equation*}
where
\begin{equation*}
\phi_k= \frac{k + \sqrt{k^2+4}}{2},
\end{equation*}
is the $k-$golden mean.
\end{remark}

The map $O_{G_k}$ has one fully supported ergodic measure which corresponds to the projection of the unique invariant measure for $O_k$.

\end{document}